\documentclass[11pt]{amsart}

\usepackage{tikz} 
\usepackage[bottom]{footmisc}               % See geometry.pdf to learn the layout options. There are lots.
\usetikzlibrary{calc,decorations.markings,arrows}
\usepackage[left=3.3cm,top=2.7cm,right=3.3cm]{geometry}               
\usepackage{graphicx}
\usepackage{subcaption}
\usepackage{amssymb}
\usepackage{amsmath,amsthm}
\usepackage{epstopdf}
\usepackage{mathtools}
\usepackage{array}
\usepackage{enumerate}
\usepackage{tikz}
\usepackage{tikz-cd}
\usepackage{cite}
\usepackage{adjustbox}
\usepackage{color}
\usepackage[colorlinks=true, linkcolor=blue, anchorcolor=blue, citecolor=blue, filecolor=blue, menucolor= blue, urlcolor=blue]{hyperref}
\usepackage{amsthm}
\theoremstyle{definition}
\newtheorem{theorem}{Theorem}[subsection]
\newtheorem*{theoremsp}{Theorem \ref{thm:main}}
\newtheorem*{theoremsp2}{Proposition \ref{prop:comm}}

\newtheorem{conjecture}[theorem]{Conjecture}
\newtheorem{lemma}[theorem]{Lemma}
\newtheorem{prop}[theorem]{Proposition}
\newtheorem{cor}[theorem]{Corollary}

\newtheorem{definition}[theorem]{Definition}

\newtheorem{example}[theorem]{Example}
\DeclareMathOperator{\Id}{Id}
\title[Matching Polynomials for Cycles Commute]{A New [Combinatorial] Proof of the Commutativity of Matching Polynomials for Cycles}

\author{Garner Cochran \textsuperscript{1,3}}
\address[Cochran]{Department of Mathematics and Computer Science, Berry College, Mount Berry, GA, 30149}

\author{Corbin Groothuis\textsuperscript{1,2}}
\address[Groothuis]{Department of Mathematics, University of Nebraska, Lincoln, NE, 68588}

\author{Andrew Herring\textsuperscript{1}}
\address[Herring]{Department of Mathematics, Western University, London, ON, Canada, N6A 5B7}
\email{aherrin6@uwo.ca}

\author{Ranjan Rohatgi\textsuperscript{1}}
\address[Rohatgi]{Department of Mathematics and Computer Science, Saint Mary's College, Notre Dame, IN, 46556}

\author{Eric Stucky\textsuperscript{1,2}}
\address[Stucky]{School of Mathematics, University of Minnesota, Minneapolis, MN, 55455}

\thanks{1--Research supported in part by National Science Foundation grants \#1604733 and \#1604458 “Collaborative Research: Rocky Mountain-Great Plains Graduate Research Workshop in Combinatorics,” \\ 2--Research supported in part by a generous grant from the Institute for Mathematics and its Applications, \\ 3--Research supported in part by National Security Agency, Workshop Grant \#H98230-16-1-0018 "The 2016 Rocky Mountain – Great Plains Graduate Workshop in Combinatorics"
}
\begin{document}
\begin{abstract}
We prove some functional equations involving the (classical) matching polynomials of path and cycle graphs and the $d$-matching polynomial of a cycle graph.  A matching in a (finite) graph $G$ is a subset of edges no two of which share a vertex, and the matching polynomial of $G$ is a generating function encoding the numbers of matchings in $G$ of each size.  The $d$-matching polynomial is a weighted average of matching polynomials of degree-$d$ covers, and was introduced in a paper of Hall, Puder, and Sawin.  Let $\mathcal{C}_n$ and $\mathcal{P}_n$ denote the respective matching polynomials of the cycle and path graphs on $n$ vertices, and let $\mathcal{C}_{n,d}$ denote the $d$-matching polynomial of the cycle $C_n$.  We give a purely combinatorial proof that $\mathcal{C}_k (\mathcal{C}_n (x)) = \mathcal{C}_{kn} (x)$ en route to proving a conjecture made by Hall: that $\mathcal{C}_{n,d} (x) = \mathcal{P}_d (\mathcal{C}_n (x))$.
\end{abstract}

\maketitle

\section{Introduction}
\subsection{} 
\textit{Ramanujan graphs} are an important class of regular graphs connecting algebraic geometry to graph theory and number theory. Given a $k$-regular graph $G$ and its adjacency matrix $M$ with eigenvalues $\lambda_{n-1}\leq \lambda_{n-2} \leq \ldots \lambda_1 \leq \lambda_0$, we say $G$ is a Ramanujan graph if $\vert \lambda_i \vert \leq 2\sqrt{k-1}$ for all $i\neq 0$. The earliest known examples of such graphs are $K_n$ and $K_{n,n}$. Most constructions of Ramanujan graphs are algebraic in nature, and arguments for infinite families were mostly limited to graphs with a prime power degree of regularity. However, recently, in \cite{Marcus}, Marcus, Spielman, and Srivastava proved that every bipartite Ramanujan graph has a Ramanujan $2$-covering graph, which showed that there were an infinite number of Ramanujan graphs of any degree of regularity. In \cite{Hall}, Hall, Puder, and Sawin generalized this result by proving that every bipartite graph without self-loops has a Ramanujan $d$-covering for every $d$. To do so, they introduced the $d$-matching polynomial.

\begin{definition}\label{DEF1}
Let $G$ be a finite, undirected graph with vertex set $V(G)$ and edge set $E(G)$.  A \textbf{matching} of $G$ is a subset $M \subseteq V(G)$ such that every $v \in V(G)$ is incident to at most one $e \in M$.  Let $a(G,i)$ denote the number of matchings $M$ of $G$ with $|M| = i$.  For a graph $G$ with $|V(G)| = n$, the \textbf{matching polynomial, $\mathcal{M}_G \in \mathbb{ Q } [x]$} is defined by 
$$
    \mathcal{M}_G (x) = \sum_{i=0}^{\lfloor \frac{n}{2} \rfloor} (-1)^i a(G,i) x^{n - 2i} 
$$
The \textbf{$d$-matching polynomial of $G$} is defined by 
$$
	\mathcal{M}_{G, d}(x) = \frac{1}{|\mathcal{L}_d (G)|} \sum_{\lambda \in \mathcal{L}_d (G)} \mathcal{M}_{G_\lambda}(x)
$$
\end{definition}
Here, $\mathcal{L}_d (G)$ is the set of maps from $E(G)$ to $S_d$, where $E(G)$ is the edge set of $G$ and $S_d$ is the set of permutations on $[d]:=\{1,2,\dots,d\}$, and $G_\lambda$ is the $d$-cover corresponding (in a sense which will be made precise in Section \ref{subsection:sdlambda}) to a given labeling $\lambda \in \mathcal{L}_d (G)$. Essentially, the $d$-matching polynomial is the average of all of the (classical) matching polynomials that come from the $d$-covers of a graph $G$. It is known due to  Heilmann and Leib in \cite{Heilmann} that the matching polynomial has only real roots when $G$ has no self-loops, and Hall et al. showed that the $d$-matching polynomial shares this important property \cite{Hall}. Independently, Hall made the following conjecture:
\begin{conjecture} \label{conj:Hall}
$$\mathcal{M}_{C_{n},d}(x) = \mathcal{M}_{P_{nd+n-1}}(x) / \mathcal{M}_{P_{n-1}}(x)$$
\end{conjecture}
\noindent where $C_\ell$ denotes the cycle graph on $\ell$ vertices and $P_\ell$ denotes the path graph on $\ell$ vertices.

In order to prove this Conjecture \ref{conj:Hall}, we will take advantage of the structure of Chebyshev polynomials, two families of orthogonal polynomials defined recursively. Let $\mathcal{P}_n(z) = \mathcal{M}_{P_n} (z)$ and $\mathcal{C}_n(z) = \mathcal{M}_{C_n} (z)$. It is well known (see \cite{Godsil}) that there is a relationship between the classical matching polynomials of the path graph and cycle graph with the Chebyshev polynomials in the following way:
$$ \mathcal{P}_n(2x)=\mathcal{U}_n(x)$$
$$ \mathcal{C}_n(2x)=2\mathcal{T}_n(x)$$
where $\mathcal{T}_n(x)$ denotes the $n$th Chebyshev polynomial of the first kind and $\mathcal{U}_n(x)$ denotes the $n$th Chebyshev polynomial of the second kind. Through these relations and general properties of the Chebyshev polynomials, we will be able to manipulate the $d$-matching polynomials combinatorially.

In Section \ref{section:Definitions}, we will provide some preliminary definitions and notation. In Section \ref{section:Comp}, we provide some results which allow us to explicitly write out the $d$-matching polynomial for the cycle graph in terms of classical matching polynomials by using permutations rather than labelings.
In Section \ref{section:Combproofs}, we provide combinatorial proofs of some identities which are analogues of Chebyshev identities. The most important of these is the following:
\begin{theoremsp2}
 $\mathcal{C}_{kn}(x) = \mathcal{C}_k(\mathcal{C}_n(x))$.

\end{theoremsp2}

We also use these identities to rewrite Conjecture \ref{conj:Hall} into the following equivalent form:
\begin{theoremsp}
$$\mathcal{C}_{n,d}(x)=\mathcal{P}_d(\mathcal{C}_n(x))$$
\end{theoremsp}

\noindent In Section \ref{section:Main}, we prove this equivalent theorem, thus proving the conjecture.

\section{Definitions}\label{section:Definitions}
\subsection{Symmetric graphs}
In this subsection we define the basic objects of our study.  At many stages in this development, proofs have required in an essential way a notion of orientation on a graph.  While the actual choice of an orientation is immaterial, its existence plays a crucial role.  See Subsection \ref{SYMM-UND}  for more discussion of this philosophy. 

\begin{definition}
A \textbf{(symmetric) graph} $G$ is a tuple $(V(G),E(G),t,h,\tau)$, where we have:
\begin{enumerate}[(1)]
\item a finite set $V(G)$ of \textbf{vertices},\vspace{0.5em}
\item a finite set $E(G)$ of \textbf{edges},\vspace{0.5em}
\item a pair of maps $(t,h) : E(G) \rightarrow V(G)^2$ called the \textbf{tail and head maps}, and \vspace{0.5em}
\item an element $\tau \in \mathrm{Sym}(E(G))$ (the symmetric group on the set $E(G)$) such that for each $e \in E(G)$:
\begin{align*}
	t (\tau (e)) = h(e) &, \  h(\tau (e)) = t(e), \ and \\
	\tau (e) &\ne e = \tau^2 (e).
\end{align*}
\end{enumerate}
\noindent Each set $\{e , \tau (e) \}$ will be referred to as a \textbf{geometric edge} (cf. Serre \cite{serre1980trees}).
\end{definition}
\vspace{0.5em}
\begin{example}
~\begin{enumerate}[(1)]
\item Fix some $n >0$.  The path graph on $n$ vertices, denoted $P_n$ is given by the following data:
\begin{align*}
	V(P_n) &= \{0 , 1, \ldots , n-1\}, \ E(P_n) = \{0^+ , 1^+ , \ldots , (n-2)^+ \} \coprod \{0^- , 1^- , \ldots , (n-2)^-\} \\[0.5em]
	(t,h)(j^+) &= (j,j+1), \ (t,h)(j^-) = (j+1 , j): \ \text{for $j=0 , \ldots, n-2$}\\[0.5em]
	\tau (j^\pm) &= j^\mp
\end{align*}
For the case $n=0$ we take the convention that $P_0$, the path graph on $0$ vertices, is the \textbf{empty graph} which has no vertices and no edges.
\item Fix $n >0$.  The cycle graph on $n$ vertices, denoted $C_n$, is given by the following data:
\begin{align*}
	V(C_n) &= \mathbb{Z} / n \mathbb{Z} = \{0 , 1, \ldots , n-1\}, \ E(C_n) = \{0^+ , 1^+ , \ldots , (n-1)^+ \} \coprod \{0^- , 1^- , \ldots , (n-1)^-\} \\[0.5em]
	(t,h)(j^+) &= (j,j+1),  \ (t,h)(j^-) = (j+1 , j): \ \text{for $j=0 , \ldots, n-1$}\\[0.5em]
	\tau (j^\pm) &= j^\mp
\end{align*}
We similarly assume by convention that $C_0$ is the empty graph.
\item Fix $n >0$.  If we consider the aforementioned graphs with edge sets given by the respective sets of geometric edges we obtain undirected versions of $C_n$ and $P_n$.  We'll denote these also by $C_n$ and $P_n$ since our philosophy will be to treat symmetric graphs and undirected graphs as essentially the same objects.  For example, 
\begin{align*}
    V(P_n) = \{0 , 1, \ldots , n-1\} &, \ E(P_n) = \{e_0, e_1 , \ldots , e_{n-2}\}; \\[0.5em]
    V(C_n) = \{0,1, \ldots , n-1\} &, E(C_n) = \{e_0 , e_1 , \ldots , e_{n-1}\}.
\end{align*}
Refer to Figure \ref{FIG1}.  Just as with their symmetric counterparts, (undirected) $P_0$ and $C_0$ are both the empty graph.
\end{enumerate}

\begin{figure*}
        \centering
        \begin{subfigure}[b]{0.475\textwidth}
            \centering
\begin{tikzpicture}
[scale=.8,auto=left,every node/.style={circle,draw,minimum size=4.5pt,inner sep=2.5pt,fill=blue!20}]
\node[label={$0$}] (n1) at (0,0) {} ;
\node[draw=none,fill=none,rectangle] (l1) at (0.75, 1) {$0^-$};
\node[label={$1$}] (n2) at (1.5,0) {};
\node[draw=none,fill=none,rectangle] (l1) at (2.25, 1) {$1^-$};
\node[label={$2$}] (n3) at (3,0) {};
\node[label={[label distance=-0.26cm]90:{$n-2$}}] (n4) at (4.5,0) {} ;
\node[draw=none,fill=none,rectangle] (l1) at (5.5, 1) {$(n-2)^-$};
\node[label={0:{$n-1$}}] (n5) at (6,0) {};
\node[draw=none,fill=none,rectangle] (l1) at (0.75, -0.75) {$0^+$};
\node[draw=none,fill=none,rectangle] (l1) at (2.25, -0.75) {$1^+$};
\node[draw=none,fill=none,rectangle] (l1) at (5.5, -0.75) {$(n-2)^+$};

    \begin{scope}[thick,decoration={
    markings,
    mark=at position 0.5 with {\arrow{>}}}
    ]
    \foreach \from/\to in {n1/n2,n2/n3,n4/n5}
    \draw[postaction=decorate] (\to) [bend right=20] to (\from);
    \end{scope}
      \begin{scope}[thick,decoration={
    markings,
    mark=at position 0.5 with {\arrow{>}}}
    ]
    \foreach \from/\to in {n1/n2,n2/n3,n4/n5}
    \draw[postaction=decorate] (\from)[bend right=20] to (\to);
    \end{scope}
\draw[dotted] (n3) to (n4);
\end{tikzpicture}
  \caption{(symmetric) $P_n $}
        \end{subfigure}
        \hfill
        \begin{subfigure}[b]{0.475\textwidth}  
            \centering
 \begin{tikzpicture}
[scale=.8,auto=left,every node/.style={circle,draw,minimum size=4.5pt,inner sep=2.5pt,fill=blue!20}]
\node[label={-45:{$0$}}] (n1) at (-45:1.75) {} ;
\node[draw=none,fill=none,rectangle] (e1) at (0:2.25) {$0^+$};
\node[label={45:{$1$}}] (n2) at (45:1.75) {};
\node[draw=none,fill=none,rectangle] (e1) at (90:2.25) {$1^+$};
\node[label={135:{$2$}}] (n3) at (135:1.75) {};
\node[label={[label distance=-0.26cm]225:{$n-1$}}] (n4) at (225:1.75) {} ;
\node[draw=none,fill=none,rectangle] (e1) at (270:2.25) {$(n-1)^+$};
\node[draw=none,fill=none,rectangle] (e1) at (0:0.6) {$0^-$};
\
\node[draw=none,fill=none,rectangle] (e1) at (90:0.6) {$1^-$};

\node[draw=none,fill=none,rectangle] (e1) at (270:0.6) {$(n-1)^-$};

    \begin{scope}[thick,decoration={
    markings,
    mark=at position 0.5 with {\arrow{>}}}
    ]
    \foreach \from/\to in {n1/n2,n2/n3,n4/n1}
    \draw[postaction=decorate] (\from) to[bend right=15] (\to);
    \end{scope}
    \begin{scope}[thick,decoration={
    markings,
    mark=at position 0.5 with {\arrow{>}}}
    ]
    \foreach \from/\to in {n1/n2,n2/n3,n4/n1}
    \draw[postaction=decorate] (\to) to[bend right=15] (\from);
    \end{scope}
    \draw[dotted] (n3) to[bend right=15] (n4);
    \draw[dotted] (n3) to[bend left=15] (n4);
\end{tikzpicture}
  \caption{(symmetric) $C_n$}
\end{subfigure}
        \vskip\baselineskip
        \begin{subfigure}[b]{0.475\textwidth}   
            \centering 
            \begin{tikzpicture}
[scale=.8,auto=left,every node/.style={circle,draw,minimum size=4.5pt,inner sep=2.5pt,fill=blue!20}]
\node[label={$0$}] (n1) at (0,0) {} ;
%\node[draw=none,fill=none,rectangle] (l1) at (0.75, 1) {$0^-$};
\node[label={$1$}] (n2) at (1.5,0) {};
%\node[draw=none,fill=none,rectangle] (l1) at (2.25, 1) {$1^-$};
\node[label={$2$}] (n3) at (3,0) {};
\node[label={[label distance=-0.26cm]90:{$n-2$}}] (n4) at (4.5,0) {} ;
%\node[draw=none,fill=none,rectangle] (l1) at (5.5, 1) {$(n-1)^-$};
\node[label={0:{$n-1$}}] (n5) at (6,0) {};
\node[draw=none,fill=none,rectangle] (l1) at (0.75, -0.75) {$e_0$};
\node[draw=none,fill=none,rectangle] (l1) at (2.25, -0.75) {$e_1$};
\node[draw=none,fill=none,rectangle] (l1) at (5.5, -0.75) {$e_{n-2}$};

    \begin{scope}[thick,decoration={
    markings,
    mark=at position 0.5 with {\arrow{>}}}
    ]
    \foreach \from/\to in {n1/n2,n2/n3,n4/n5}
    \draw (\to) to (\from);
    \end{scope}
\draw[dotted] (n3) to (n4);
\end{tikzpicture}
            \caption{(undirected) $P_n$}
        \end{subfigure}
        \quad
        \begin{subfigure}[b]{0.475\textwidth}   
            \centering 
            \begin{tikzpicture}
[scale=.8,auto=left,every node/.style={circle,draw,minimum size=4.5pt,inner sep=2.5pt,fill=blue!20}]
\node[label={-45:{$0$}}] (n1) at (-45:1.75) {} ;
\node[draw=none,fill=none,rectangle] (e1) at (0:2.25) {$e_0$};
\node[label={45:{$1$}}] (n2) at (45:1.75) {};
\node[draw=none,fill=none,rectangle] (e1) at (90:2.25) {$e_1$};
\node[label={135:{$2$}}] (n3) at (135:1.75) {};
\node[label={[label distance=-0.26cm]225:{$n-1$}}] (n4) at (225:1.75) {} ;
\node[draw=none,fill=none,rectangle] (e1) at (270:2.25) {$e_{n-1}$};
%\node[draw=none,fill=none,rectangle] (e1) at (0:0.6) {$0^-$};
\
%\node[draw=none,fill=none,rectangle] (e1) at (90:0.6) {$1^-$};

%\node[draw=none,fill=none,rectangle] (e1) at (270:0.6) {$(n-1)^-$};
    \begin{scope}[thick,decoration={
    markings,
    mark=at position 0.5 with {\arrow{>}}}
    ]
    \foreach \from/\to in {n1/n2,n2/n3,n4/n1}
    \draw (\from) to[bend right=15] (\to);
    \end{scope}
    \draw[dotted] (n3) to[bend right=15] (n4);
\end{tikzpicture}
            \caption{(undirected) $C_n$}
        \end{subfigure}
        \caption{paths and cycles}
    \end{figure*}
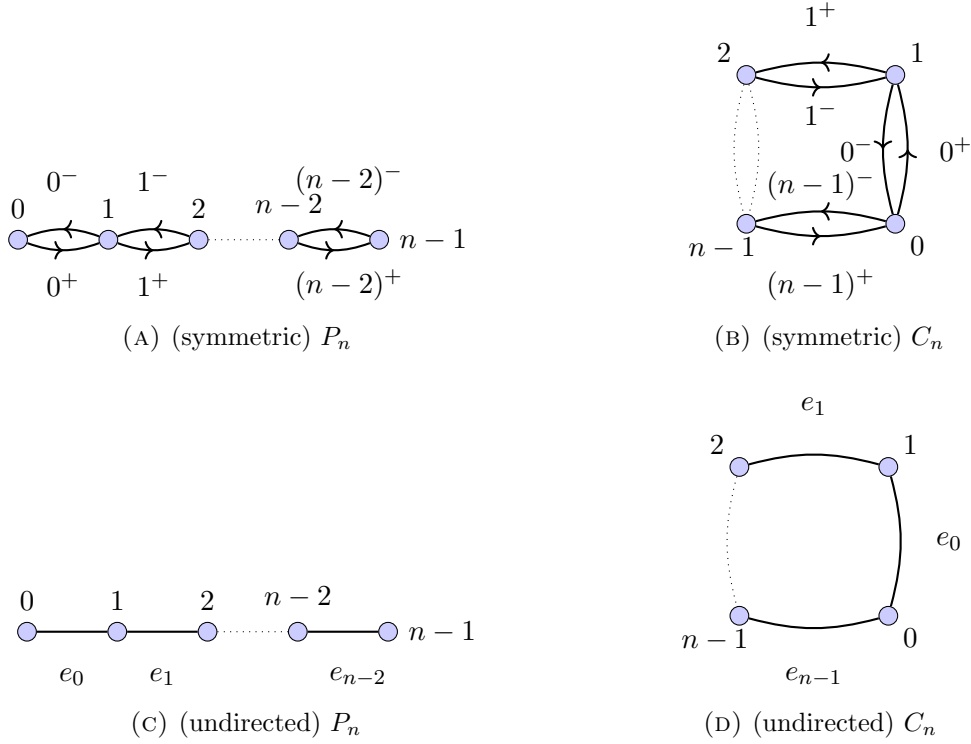\label{FIG1}
\end{example}

\begin{definition}
Let $G_j = (V(G_j) , E(G_j) , E(G_j) \xrightarrow{(t,h)} V(G_j)^2 , \tau_j))$ be graphs for $j= 1,2$.  Then a \textbf{morphism $f: G_1 \to G_2$} consists of a pair of maps
\begin{align*}
	f_E: E(G_1) &\to E(G_2) \\ \vspace{0.5em}
	f_V : V(G_1) &\to V(G_2)
\end{align*}
such that each of the following diagrams commutes:
$$
\begin{tikzcd}
E(G_1) \ar{r}{f_E} \ar{d}[swap]{(t,h)} & E(G_2) \ar{d}{(t,h)} & & E(G_1) \ar{r}{f_E} \ar{d}[swap]{\tau_1} & E(G_2) \ar{d}{\tau_2} \\
V(G_1)^2 \ar{r}[swap]{f_V ^2} & V(G_2)^2 & & E(G_1) \ar{r}[swap]{f_E} & E(G_2)
\end{tikzcd}
$$
An \textbf{isomorphism} from $G_1$ to $G_2$ is a morphism $f$ such that each of the set maps $f_V$ and $f_E$ are bijections.
\end{definition}
The obvious example is the \textbf{identity morphism}: $\Id_{G}$, which has vertex map (respectively edge map) $\mathrm{Id_{V(G)}}$ (respectively $\mathrm{Id_{E(G)}}$ ) which is the identity set map on $V(G)$ (respectively E(G)).  We remark that we could equivalently have defined an isomorphism as follows: an isomorphism from $G_1$ to $G_2$ is a morphism $f = (f_V , f_E) : G_1 \to G_2$ such that there exists a morphism $f^{-1} = (f_V^{-1}, f_E ^{-1}) : G_2 \to G_1$ such that we have the following:
\begin{align*}
    f_V ^{-1} \circ f_V = \mathrm{Id_{V(G_1)}}&, \ f_V \circ f_V ^{-1} = \mathrm{Id_{V(G_2)}} \\[0.5em]
    f_E ^{-1} \circ f_E = \mathrm{Id_{E(G_1)}}&, \ f_E \circ f_E ^{-1} = \mathrm{Id_{E(G_2)}}
\end{align*}
In this case we'll simply write $f^{-1} \circ f = \mathrm{Id_{G_1}}$ and $f \circ f^{-1} = \mathrm{Id_{G_2}}$.
\subsection{Covers}
For this subsection fix a (symmetric) graph $G=(V , E, t, h, \tau)$.  
\begin{definition}
For any $v \in V$, the \textbf{neighborhood of $v$}, denoted $N_v$, consists of one third of each edge in $t^{-1}(v) \cup h^{-1} (v)$.  
\end{definition}
\noindent Frequently we will write $e \in N_v$ to mean that $e \in t^{-1} (v) \cup h^{-1} (v)$ for a given edge $e \in E$.  
\begin{definition}
$G$ is \textbf{connected} if for any pair of vertices $x$ and $y$ there exists a number $n \ge 0$ and a morphism $f : P_n \to G$ such that $f(0) = x$ and $f(n-1) = y$.
\end{definition}
\noindent In other words, $G$ is connected if there is a path between any two of its vertices.
\begin{definition}\label{COV-DEF}
A \textbf{cover} of $G$ is a pair $(H,\phi)$ such that the following hold:
\begin{enumerate}[(1)]
\item $H$ is a (symmetric) graph; \vspace{0.5em}
\item $\phi : H \to G$ is a surjective morphism (i.e., each of $\phi_V$ and $\phi_E$ is a surjective map of sets); \vspace{0.5em}
\item for each $\hat{w} \in V(H)$, $\phi|_{N_{\hat{w}}} : N_{\hat{w}} \to N_{\phi (\hat{w})}$ is a bijection.
\end{enumerate}
\end{definition}
If $G$ is connected, it follows from $(3)$ in Definition \ref{COV-DEF} that $|\phi^{-1}(v)| = |\phi^{-1}(w)|$ for every $v,w \in V(G)$.  In this case, if $|\phi^{-1}(v)| = d$, then we say that \textbf{$H$ is a degree $d$ cover of $G$} or sometimes just that \textbf{$H$ is a $d$-cover}.  
\vspace{0.5em}
\begin{example}\label{EX1}
~\begin{enumerate}[(1)]
\item Assume $G$ is connected.  Then we have the \textbf{trivial cover}: $H = G$, $\phi = \Id_G$.  In this case $H$ is a $1$-cover of $G$ and $\phi$ is an isomorphism.
\item Let $n,\ell$ be positive integers and consider the morphism $\pi_{\ell} : C_{\ell \cdot n} \to C_n$ defined by 
\begin{align*}
	\pi_{\ell} (j) &= j (\mathrm{mod}\  n) \\[0.5em]
	\pi_{\ell} (j^{\pm}) &= j^{\pm} (\mathrm{mod}\  n)
\end{align*}
for each $j = 0 , \ldots , \ell n - 1$.  Then $(C_{\ell \cdot n} , \pi_{\ell})$ is a degree $\ell$ cover of $C_n$.
\item Let $G = C_n$ and consider a tuple $(\mu_1 , \ldots , \mu_k)$ of non-negative integers.  Let
$$
    H:= \coprod_{i=1} ^k C_{n \cdot \mu_i}
$$
and let $\phi$ be the obvious morphism whose restriction to each factor in the disjoint union is the map $\pi_{\mu_i} : C_{n \cdot \mu_i} \to C_n$ discussed above.  If we set $d := \sum_{i=1} ^k \mu_i$, then $(H,\phi)$ is a $d$-cover of $G$.
\end{enumerate}
\end{example}

\subsection{$S_d$-Labelings}\label{subsection:sdlambda}
Fix $d >0$ and let $[d]:= \{1 , \ldots , d\}$.  We let $S_d$ denote the group of permutations on the set $[d]$.  We'll prefer to write the action of $S_d$ on $[d]$ on the right so that $i ^ \sigma$ denotes the image of $i$ under the permutation $\sigma \in S_d$.  
\begin{definition}
An \textbf{$S_d$-labeling on $G$} is a set map $\sigma : E(G) \to S_d$ such that 
$$
	\sigma (\tau (e)) = \sigma (e)^{-1}
$$
for every $e \in E(G)$.  Let $\mathcal{L}_d (G)$ denote the set of all $S_d$-labelings on $G$.  
\end{definition}
Given an $S_d$-labeling $\sigma$, we now describe how to construct a $d$-cover $(G_\sigma , \phi)$ of $G$.  Let
\begin{align*}
V(G_\sigma):= V(G) \times [d], \ \phi (v,j) = v&; \ \ \  E(G_\sigma):= E(G) \times [d], \ \phi (e,j) = e \\[0.5em] 
t(e,j) = (t(e) , j)&, \     h(e,j) = (h(e) , j^{\sigma(e)})\\[0.5em]
\tau (e,j) &= (\tau (e) , j^{\sigma(e)})
\end{align*}
for all $j = 1, \ldots , d$, for all $e \in E(G)$, and for all $v \in V(G)$.
\\ Some care is needed in examining the above definitions: for example when we define $t(e,i) = (t(e) , i)$, on the left hand side we refer to the tail map on $G_\sigma$ and on the right to the tail map on $G$.

\section{A Computational Improvement}\label{section:Comp}
To simplify notation, we let $C_{n,\lambda}$ denote the $d$-cover of $C_n$ determined by $\lambda \in \mathcal{L}_d (C_n)$ as described above (i.e., $G=C_n$ and $\sigma = \lambda$). 
\\ Many different $S_d$-labelings $\lambda$ on $C_n$ yield isomorphic covers $C_{n,\lambda}$.  In this section we show how to group such labelings together to increase the efficiency with which one computes $\mathcal{M}_{C_n , d} (x)$.

\subsection{Conjugacy and isomorphism}
\begin{lemma}\label{LEM1}
Fix $\lambda , \mu \in \mathcal{L}_d (C_n)$, fix some $i_0 \in V(C_n) = \mathbb{Z} / n \mathbb{Z}$ and let $\ell_0:= i_0 + 1$.  Suppose that the following hold:
\begin{enumerate}[(i)]
\item $\lambda( i ^+) = \mu (i^+)$ for all $i \ne i_0 , \ell_0$; \vspace{0.5em}
\item $\lambda(i_0^+)\lambda (\ell_0^+)=\mu(i_0^+)\mu (\ell_0^+)$. \vspace{0.5em}
\end{enumerate}
Then there is an isomorphism $f : C_{n,\lambda} \cong C_{n , \mu}$.  
\end{lemma}
The proof is routine, except for the definition of the isomorphism. Let 
$$
	f_V (i,j) = \left\{ \begin{array}{ll} (i,j) & i \ne \ell_0; \\[0.5em] 
	(i, j^\delta) & i = \ell_0; \end{array} \right. \ \ 
	f_E (i^{+},j) = \left\{ \begin{array}{ll} (i^{+},j) & i \ne \ell_0 ;\\[0.5em] 
	(i^{+}, j^\delta) & i = \ell_0; \end{array} \right.
$$
where $\delta := \lambda (i_0 ^+)^{-1} \mu (i_0 ^+)$.  We define $f_E (i^- , j)$ so as to guarantee that $\tau_{\mu} \circ f_E = f_E \circ \tau_{\lambda}$ and find that imposing this condition gives
$$
	f_E (i^{-},j) = \left\{ \begin{array}{ll} (i^{-},j); & i \ne i_0 \\[0.5em] 
	(i^{-}, j^\delta); & i = i_0 \end{array} \right.
$$
The remainder of the proof, which we omit, is a long and uninsightful computation to check that $f$ as defined above gives an isomorphism.

\begin{cor}
If $\prod_{i=0} ^{n-1} \lambda (i^+) = \prod_{i=0} ^{n-1} \mu (i^+)$, then the associated covers $C_{n,\lambda}$ and $C_{n,\mu}$ are isomorphic.
\end{cor}
\begin{proof}
We define an equivalence relation $\sim$ on $\mathcal{L}_d (C_n)$ by $\gamma_1 \sim \gamma_2$ if and only if
$$
	\prod_{i=0} ^{n-1} \gamma_1 (i^+)  = \prod_{i=0} ^{n-1} \gamma_2 (i^+) 
$$
Then we define a sequence of labelings $\lambda = \lambda_0 , \lambda_1 , \ldots , \lambda_{n-1}$ by 
$$
	\lambda_k (i^+) = \left\{ \begin{array}{ll} 1_{S_d}, & i=0,\ldots , k-1; \\[0.5em]
	\prod_{i=0} ^k \lambda(i^+), & i=k; \\[0.5em]
	\lambda (i^+), & i= k+1 , \ldots , n-1.
	\end{array}
	\right.
$$
One easily checks that $\lambda_k \sim \lambda_{k+1}$ for each $k = 0, \ldots , n-2$.  We claim that this induces a sequence of isomorphisms on the associated covers $C_{n,\lambda_k} \cong C_{n , \lambda_{k+1}}$ and thus $C_{n , \lambda_0} \cong C_{n,\lambda_{n-1}}$ by transitivity.  Indeed, we apply Lemma \ref{LEM1} with $i_0 = k$: because 
\begin{align*}
    \lambda_k (i^+) &= \lambda_{k+1} (i^+) = \left\{ \begin{array}{ll} 1_{S_d} & i <i_0; \\[0.5em] 
    \lambda(i^+) & i >\ell_0; \end{array} \right. \\[0.5em]
    \lambda_k (i_0^+) \lambda_k (\ell_0^+) &= \prod_{i=0}^{\ell_0} \lambda(i^+) = \lambda_{k+1} (i_0 ^+) \lambda_{k+1} (\ell_0 ^+)
\end{align*}
we see that the hypotheses of Lemma \ref{LEM1} are met and thus obtain an isomorphism $C_{n,\lambda_k} \cong C_{n,\lambda_{k+1}}$.  Of course we can do the same thing for $\mu$: define a sequence $\mu = \mu_0 \sim \mu_1 \sim \ldots \sim \mu_{n-1}$ by 
$$
	\mu_k (i^+) = \left\{ \begin{array}{ll} 1_{S_d}, & i=0,\ldots , k-1; \\[0.5em]
	\prod_{i=0} ^ k \mu(i^+), & i=k; \\[0.5em]
	\mu (i^+), & i= k+1 , \ldots , n-1.
	\end{array}
	\right.
$$
and we thus obtain two sequences of isomorphisms:
\begin{align*}
	C_{n,\lambda} \cong C_{n,\lambda_1} &\cong \ldots \cong C_{n,\lambda_{n-1}} \\[0.5em]
	C_{n,\mu} \cong C_{n,\mu_1} &\cong \ldots \cong C_{n,\mu_{n-1}}
\end{align*}
Our final task is to prove that the relation $\lambda_{n-1} \sim \mu_{n-1}$, which holds by hypothesis, also induces an isomorphism on corresponding covers.  This is achieved by applying Lemma \ref{LEM1} to $\lambda_{n-1}$ and $\mu_{n-1}$ with $i_0 = n-2$.  We have 
\begin{align*}
	\lambda_{n-1}(i^+) &= 1_{S_d} = \mu_{n-1} (i^+), \ \text{for $i\ne i_0 , \ell_0$} \\[0.5em]
	\lambda_{n-1} (i_0^+) \lambda_{n-1} (\ell_0^+) &= \prod_{i=0}^{\ell_0} \lambda(i^+) \\[0.5em]
	 &= \prod_{i=0}^{\ell_0} \mu(i^+) =\mu_{n-1} (i_0^+) \mu_{n-1} (\ell_0^+) 
\end{align*}
So by the lemma, there's an isomorphism $C_{n,\lambda_{n-1}} \cong C_{n,\mu_{n-1}}$.  So by composing our isomorphism sequences:
$$
	C_{n,\lambda} \cong \ldots \cong C_{n,\lambda_{n-1}} \cong C_{n,\mu_{n-1}} \cong \ldots \cong C_{n,\mu}
$$
we obtain the desired result.
\end{proof}
\begin{lemma}
Let $\lambda \in \mathcal{L}_d (C_n)$ and $g \in S_d$.  Define $\lambda ^g \in \mathcal{L}_d (C_n)$ by 
$$
	\lambda^g (i^+) = (\lambda (i^+))^g = g^{-1} \lambda (i^+) g
$$
Then $C_{n,\lambda} \cong C_{n,\lambda^g}$.
\end{lemma}
\begin{proof}
We define a map $f = (f_V , f_E) : C_{n,\lambda} \to C_{n,\lambda^g}$ by 
\begin{align*}
	f_V (i,j) &= (i , j^g) \\[0.5em]
	f_E (i^{\pm} , j) &= (i^{\pm} , j^g)
\end{align*}
for all $i \in V(C_n)$ and all $j \in \{1 , \ldots , d\}$, and apply the definitions of the involved pieces to show that $f$ is an isomorphism.
\end{proof}

\begin{cor}\label{CongCor}
Let $\lambda , \mu \in \mathcal{L}_d (C_n)$.  Suppose that there exists $g \in S_d$ such that 
$$
	g^{-1}\left(\prod_{i=0} ^{n-1} \lambda(i^+)\right) g = \prod_{i=0}^{n-1} \mu(i^+).
$$
Then $C_{n,\lambda} \cong C_{n,\mu}$.
\end{cor}
\begin{proof}
Immediate from the previous two results.
\end{proof}

\subsection{Symmetric graphs vs. undirected graphs}\label{SYMM-UND}
At certain stages of the development, it was useful to explicitly refer to an orientation on our graphs.  This is the essential reason for employing ``symmetric graphs'' as opposed to undirected graphs.  Notice however that all of the statements we've made for symmetric graphs hold for their undirected counterparts: an undirected graph is sent to a symmetric graph by ``splitting'' each edge into a pair of edges (one running each direction), and from a symmetric graph one obtains an undirected graph by simply conisdering the geometric edges (i.e. collapsing the pair of directed edges to a single undirected edge).  All of the isomorphisms we've discussed respect these two procedures.  

\subsection{Cycle types and isomorphism classes} \label{subsection:Comp}

Each element $\sigma \in S_d$ can be uniquely written (up to reordering) as a product of disjoint cycles of various lengths \cite{DF}.  This is called the \textbf{disjoint cycle decomposition of $\sigma$}.  The \textbf{cycle type} of $\sigma$ is the tuple $\mu=(\mu_1,\dots,\mu_d)$, where $\mu_i$ is the number of cycles of length $i$ in the disjoint cycle decomposition of $\sigma$: we write $\mathrm{cyc}(\sigma) = \mu$. Recall that the cycle type of a permutation completely determines its conjugacy class; $\sigma_1$ and $\sigma_2$ in $S_d$ are conjugate if and only if they have the same cycle type \cite{DF}.  

We finally arrive at our computational improvement for computing the $d$-matching polynomial of $C_n$.  For notational convenience, let $\mathcal{C}_m (x) = \mathcal{M}_{C_m} (x)$, the matching polynomial of $C_m$, and $\mathcal{C}_{n,d} (x) = \mathcal{M}_{C_n , d} (x)$ the $d$-matching polynomial of $C_n$.

\begin{prop}\label{Prop:WARMUP}
For any positive integers $n$ and $d$,
$$
	\mathcal{C}_{n,d} (x) = \frac{1}{d!} \sum_{\substack{\sigma \in S_d \\ \mathrm{cyc}(\sigma) =  \mu}} \prod_{i=1}^d \mathcal{C}_{n \cdot \mu_i} (x),
$$
where the summation is over all $\sigma\in S_d$.
\end{prop}

\begin{proof}
We recall that by Definition \ref{DEF1},
$$ 
	\mathcal{C}_{n,d} (x) = \frac{1}{|\mathcal{L}_d (C_n) |} \sum_{\lambda \in \mathcal{L}_d (C_n)} \mathcal{M}_{C_{n,\lambda}} (x).
$$
To each $\lambda \in \mathcal{L}_d (C_n)$ we associate the permutation 
$$
	\lambda (C_n) := \prod_{i=0}^{n-1} \lambda(i^+)
$$
By Corollary \ref{CongCor}, we can identify the isomorphism class of $C_{n,\lambda}$ with the conjugacy class in $S_d$ containing $\lambda(C_n)$. Proposition \ref{Prop:WARMUP} therefore follows from the observation that if $\mathrm{cyc}(\lambda (C_n)) = \mu = (\mu_1, \ldots , \mu_d)$, then 
$
	C_{n,\lambda} \cong \coprod _{i=1}^d C_{n\cdot \mu_i}
$
(cf. item $(3)$ in Example \ref{EX1}).  
\end{proof}

\section{Combinatorial Proofs for Chebyshev Identities}\label{section:Combproofs}

Let $\mathcal{C}_n (x)$ denote $\mathcal{M}_{C_n} (x)$, $\mathcal{P}_n (x)$ denote $\mathcal{M}_{P_n} (x)$, and $\mathcal{C}_{n,d} (x)$ denote $\mathcal{M}_{C_{n,d}} (x)$. As is shown in \cite{Godsil}, there are important relationships between the matching polynomials $\mathcal{C}_n(x)$ and $\mathcal{P}_n(x)$ and the Chebyshev polynomials of the first and second kind respectively.

\begin{definition}
	The Chebyshev polynomials of the first kind, $T_n(x)$ are defined by the following recurrence relation:
	\begin{itemize}
		\item $T_0(x)=1$,
		\item $T_1(x)=x$, and 
		\item $T_{n+1}(x)=2xT_n(x)-T_{n-1}(x)$.
	\end{itemize}
\end{definition}

\begin{definition}
	The Chebyshev polynomials of the second kind, $U_n(x)$ are defined by the following recurrence relation:
	\begin{itemize}
		\item $U_0(x)=1$,
		\item $U_1(x)=2x$, and 
		\item $U_{n+1}(x)=2xU_n(x)-U_{n-1}(x)$.
	\end{itemize}
\end{definition}

\begin{lemma}
	$\frac{1}{2}\mathcal{C}_n(2x)=T_n(x)$ and $\mathcal{P}_n(2x)=U_n(x)$
\end{lemma}
For a proof of these facts, see chapters $1$ and $8$ of \cite{Godsil} where the ``three term recurrences'' are used to prove the above Lemma.

\subsection{Commutativity} \label{subsection:Commute}

The Chebyshev polynomials of the first kind also have a special relationship with respect to composition.
$$ 
    T_{kn}(x) = T_k(T_n(x)).
$$
In particular, this implies that the Chebyshev polynomials are commuting operators. We may substitute $T_n(x)=\frac12\mathcal{C}_n(2x)$ to obtain the following combinatorial statement:
$$
    \mathcal{C}_{kn}(x) = \mathcal{C}_k(\mathcal{C}_n(x))
$$

In this section we provide a combinatorial proof of this statement, based off a similar proof given by Walton in \cite{Baddad}. 
Our main goal is to show $\mathcal{C}_{kn}=\mathcal{C}_k(\mathcal{C}_n)$, but first we must algebraically reduce this to a combinatorial identity for which we can give a proof. Given a list of integers $\alpha=(\alpha_1,\dots,\alpha_\ell)$, let $S(\alpha)$ denote their sum, write $\alpha\vDash S(\alpha)$, and denote the length of the list by $|\alpha|$. We then have:
\begin{align*}
	\mathcal{C}_k(\mathcal{C}_n)&=\sum_{i\geq0}(-1)^ia(C_k,i)\left(\mathcal{C}_n\right)^{k-2i}\\
	&=\sum_{i\geq0}(-1)^ia(C_k,i)\left(\sum_{\ell \geq 0}(-1)^\ell a(C_n,\ell)x^{n-2\ell}\right)^{k-2i}\\
	&=\sum_{i\geq0}(-1)^ia(C_k,i)\left(\sum_{\vert \alpha \vert =k-2i}(-1)^{S(\alpha)}\left(\prod_{j=1}^{k-2i}a(C_n,\alpha_j)\right)x^{(k-2i)n-2S(\alpha)}\right).
\end{align*}
Breaking the innermost sum apart by $S(\alpha)$, we make the substitution $m=in+S(\alpha)$.
\begin{align*}
	\mathcal{C}_k(\mathcal{C}_n)&=\sum_{i\geq 0}(-1)^{i}a(C_k,i)\sum_{m\geq in}\left(x^{nk-2m}\sum_{\substack{\alpha \vDash m-in\\\vert \alpha \vert = k-2i}}(-1)^{m-in}\prod_{j=1}^{k-2i}a(C_n,\alpha_j)\right)\\
	&=\sum_{m\geq 0}\left(\sum_{i=0}^{\left\lfloor \frac{m}{n}\right\rfloor}(-1)^{i(1-n)}a(C_k,i)\left(\sum_{\substack{\alpha \vDash m-in\\\vert \alpha \vert = k-2i}}\prod_{j=1}^{k-2i}a(C_n,\alpha_j)\right)\right)(-1)^mx^{nk-2m}
\end{align*}

Let $\ell G$ denote the $\ell$-fold disjoint union $\coprod_{i=1} ^\ell G$ for any positive integer $\ell$ and any graph $G$.  Notice that the innermost sum counts matchings of size $m-in$ on $(k-2i)C_n$. Moreover, there are no such matchings when $m\geq in$, so we can remove the upper index on the sum over $i$.
\begin{align*}
	\mathcal{C}_k(\mathcal{C}_n)&=\sum_{m\geq 0}\left(\sum_{i\geq 0}(-1)^{i(1-n)}a(C_k,i)a((k-2i)C_n,m-in)\right)(-1)^mx^{nk-2m}\tag{$*$}
\end{align*}

Interpreting the coefficients above: when $i$ edges are matched in $C_k$, we only consider $(k-2i)$ copies of $C_n$, and only include $m-2i$ edges in the matching. This suggests the following definitions.

\begin{definition}
For fixed $n$ and $k$, a \textbf{metacycle} is a $k$-cycle where each vertex represents an $n$-cycle. A \textbf{matching} of a metacycle is a matching of the $k$-cycle together with, for each vertex not contained in any such edge, a matching of the corresponding $C_n$. 

Denote by $e_k$ the number of edges of the matching in the $k$-cycle and by $e_n$ the number of edges of the matching in the $n$-cycles. Then the \textbf{weight} of a metacycle matching is $e_n + ne_k$, and the \textbf{sign} of a metacycle matching is $(-1)^{(1-n)e_k}$.
\end{definition}

With this terminology, we see that the coefficients of $\mathcal{C}_k(\mathcal{C}_n)$ are either a sum, or an alternating sum, over matchings of a metacycle. The remainder of this section is devoted to showing that these sums also count matchings in $C_{kn}$.

\begin{prop}\label{prop:comm}
 $\mathcal{C}_{kn}(x) = \mathcal{C}_k(\mathcal{C}_n(x))$.
\end{prop}

\begin{proof}
Our $C_{kn}$ will be viewed as $k$ rows of $n$ vertices, numbered from the top, with the edge between row $i$ and $i+1$ from $v_{i,n}$ to $v_{i+1,1}$. We call these edges \textbf{zags}. In the copies of $C_n$ in the metacycle, we will call the edge from $v_1$ to $v_n$ a \textbf{hop}.
The relationship between the two objects can be seen in the picture below.

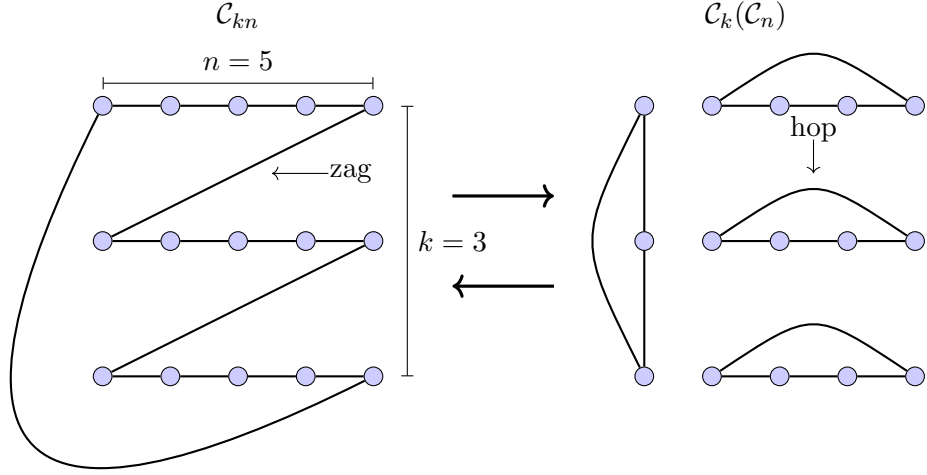
\begin{center}
\begin{tikzpicture}[scale=.6,auto=left,every node/.style={circle,draw,minimum size=4.5pt,inner sep=2.5pt,fill=blue!20}]
%C_{kn}
\node (n0) at (0,0) {} ;
\node (n1) at (1.5,0) {} ;
\node (n2) at (3,0) {} ;
\node (n3) at (4.5,0) {} ;
\node (n4) at (6,0) {} ;
\node (m0) at (0,3) {} ;
\node (m1) at (1.5,3) {} ;
\node (m2) at (3,3) {} ;
\node (m3) at (4.5,3) {} ;
\node (m4) at (6,3) {} ;
\node (p0) at (0,6) {} ;
\node (p1) at (1.5,6) {} ;
\node (p2) at (3,6) {} ;
\node (p3) at (4.5,6) {} ;
\node (p4) at (6,6) {} ;
\foreach \from/\to in {n0/n1,n1/n2,n2/n3,n3/n4,m0/m1,m1/m2,m2/m3,m3/m4,p0/p1,p1/p2,p2/p3,p3/p4,n0/m4,m0/p4}
\draw[thick] (\from) to (\to);
\draw[thick] (p0)..controls (-4,-2) and (-2,-4) ..(n4);
\draw[|-|] (0,6.5) to (6,6.5);
\node[draw=none,fill=none] at (3,7) {$n=5$};
\draw[|-|] (6.75,0) to (6.75,6);
\node[draw=none,fill=none] at (7.75,3) {$k=3$};
\node[draw=none,fill=none] at (3,8) {$\mathcal{C}_{kn}$};
\node[draw=none,fill=none] at (5.5,4.5) {zag};
\draw[->] (5,4.5) to (3.75,4.5);
%C_k (C_n)
\node (a0) at (0+12,0) {} ;
\node (a1) at (1.5+12,0) {} ;
\node (a2) at (3+12,0) {} ;
\node (a3) at (4.5+12,0) {} ;
\node (a4) at (6+12,0) {} ;
\node (b0) at (0+12,3) {} ;
\node (b1) at (1.5+12,3) {} ;
\node (b2) at (3+12,3) {} ;
\node (b3) at (4.5+12,3) {} ;
\node (b4) at (6+12,3) {} ;
\node (c0) at (0+12,6) {} ;
\node (c1) at (1.5+12,6) {} ;
\node (c2) at (3+12,6) {} ;
\node (c3) at (4.5+12,6) {} ;
\node (c4) at (6+12,6) {} ;
\foreach \from/\to in {a1/a2,a2/a3,a3/a4,b1/b2,b2/b3,b3/b4,c1/c2,c2/c3,c3/c4,a0/b0,b0/c0}
\draw[thick] (\from) to (\to);
\draw[thick] (a1)..controls (15.75,1.5)  ..(a4);
\draw[thick] (b1)..controls (15.75,4.5)  ..(b4);
\draw[thick] (c1)..controls (15.75,7.5)  ..(c4);
\draw[thick] (a0)..controls (10.5,3)  ..(c0);
\node[draw=none,fill=none] at (15.75,5.5) {hop};
\draw[->] (15.75,5.25) to (15.75,4.5);
\node[draw=none,fill=none] at (14.25,8) {$\mathcal{C}_k (\mathcal{C}_n)$};
%Bijection arrows
\draw[very thick,->] (7.75,4) to (10,4);
\draw[very thick,->] (10,2) to (7.75,2);
\end{tikzpicture}
\end{center}

%Case 1: $n$ is odd.\\
Let $M$ be a matching of size $m$ in $C_{kn}$. We will decompose the vertex set of $C_{kn}$ based on our matching. A \textbf{zag component} is the subgraph induced by vertices $Z$, where $Z$ is the set of vertices in a maximal collection of adjacent rows where all internal zags are present in our matching. If there is a row in our graph without an incoming or outgoing matched zag, then the row is a zag component on its own. Note that except in the case where all zags are in our matching, we can shift our cycle so that the first zag component starts on the first row.

Given a zag component, we will mark edges to include in a metacycle matching. There are three types of zag components:
\begin{enumerate}
\item If a zag component consists of a single row $i$: we mark all the edges in the $C_n$, represented by vertex $i$ of the metacycle, which correspond to the edges of $M$ in row $i$.
\item If a zag component from row $i$ to row $j$ has all vertices in rows $i$ and $j$ completely matched, we mark the edge in the metacycle from $j-1$ to $j$. Then for $\ell \in [i+1,j-1]$, we mark all the edges on the $C_n$, represented by vertex $\ell -1 $ of the metacycle, which correspond to the edges of $M$ in row $\ell$. Also, we mark the hop in that $C_n$.
\item If a zag component from row $i$ to row $j$ does not completely match rows $i$ and $j$ for $\ell \in [i+1,j-1]$, we again mark the edges on the $C_n$ represented by vertex $\ell -1$ of the metacycle, which correspond to the edges of $M$ in row $\ell$. Again, we also mark the hop in that $C_n$. 

Now let $t$ be the smallest index $1\leq t \leq n-1$ such that $v_{i,t} \sim v_{i,t+1}$ and $v_{j,t} \sim v_{j,t+1}$ are not edges in our matching. We modify rows $i$ and $j$ of $C_{kn}$ by swapping the tails of these rows to the right of vertex $t$. We then mark the edges on the $C_n$, represented by vertex $j$ in the metacycle, which correspond to the edge of $M$ in the modified row $i$. Finally, we mark the edges on the $C_n$, represented by vertex $j-1$ in the metacycle, which correspond to the edge of $M$ in the modified row $j$; and also the hop in that $C_n$.
\end{enumerate}

If there is only one zag component (i.e. each zag edge is in our matching), we can choose row $i$ to be $1$, row $j$ to be $k$, and then treat this as case (3).

This produces a matching of the correct weight in our metacycle. To see this is a matching, we only need check that there are no incident edges in each $C_n$ (since by construction two edges in the $C_k$ cannot be adjacent). Since each row on the interior of a zag component has an incoming and outgoing matched zag edge, $v_{\ell, 1}$ and $v_{\ell, n}$ are free to be matched by a hop in case (2). In case (3), the tail-swapping swaps any potential edge including $v_{j,n}$, and $v_{j,1}$ is matched with a zag in $C_{kn}$. Hence, we are again free to match it with a hop in the metacycle. 

To verify that this matching has the right weight, we need to show that the weight of the metacycle matching is the same as the size of $M$. Cases (1) and (3) are bijective on edges, and case (2) replaces $n-1$ edges in rows $i$ and $j$ and the edge $v_{j-1,n}v_{j,1}$ with an edge of weight $n$ on the $C_k$. 

Unfortunately, this procedure is not always reversible. However, we claim that for any metacycle matching $\widetilde M$ of weight $m$, either the procedure can be reversed, or else $\widetilde M$ can be matched with a metacycle matching of weight $m\pm 1$.

We can label each cycle $i$ with either an $H$ if a hop is taken, $r$ is no hop is taken, or $\mu$ if either $i\sim i+1$ or $i-1 \sim i$ is an edge of $\widetilde M$ in the $C_k$. Consider the string for $\widetilde M$ consisting of these letters. Since each $\mu$ must come in a pair, the language corresponding to these strings is generated by $\{H,r,\mu\mu\}$. 

\begin{lemma}
The cyclic strings of length $k$ generated by $\{H,r,\mu\mu\}$ are also generated by $(\{H^a\mu\mu: a\geq 0\}\cup \{H^ar: a\geq 0\}\cup{H^k})$. The second set of generators form a prefix free language, hence form a uniquely decodable code.
\end{lemma}
\begin{proof}
$r$ and $MM$ are in our set by setting $a=0$ and the only way we can't attach $H$'s to them is if our whole string is $H$'s. These generators are clearly prefix free.
\end{proof}

Except in the case where our string is $H^k$, we can shift our cycles again so that cycle $k$ is labeled with the second $\mu$ in a pair or by an $r$. By the lemma, we can uniquely decode the string into generators of the form $H^ar, H^k,$ or $H^a\mu\mu$. 

Due to the difference in matchings of $P_{2k}$ versus $P_{2k+1}$, we will consider cases by parity.
Odd case:\\
When $n$ is odd, each of these generators corresponds to one of the three zag components described earlier. Type (1) corresponds to $H^0r$, Type (2) corresponds to $H^a\mu\mu$, and $\{H^ar, a>0\}$ and $H^k$ correspond to Type (3). These three cases are also all in the image of our described matching above. First, we can perfectly match the top and bottom row of a zag component to get case (2). Moreover, since $n$ is odd, when we don't have a perfect matching we always have a vertex after which to tail-swap, so the opposite procedure for Type (3) is well-defined.

Even case:\\
When $n$ is even, there are two types of metacycle matchings which are not constructed by applying the Cases above:
\begin{enumerate}
\item $\mu\mu$ appears in the string
\item There is a hop component of the form $H^ar$ with first and last cycle completely matched.
\end{enumerate} 

We only see $\mu\mu$ in the string if there is a zag component with top and bottom row perfectly matched. When $n$ is even, this is impossible, since the outgoing zag from the first row leaves an odd number of vertices to be perfectly matched. In the second case, if the first and last cycle are completely matched, there is no vertex $t$ from which we can swap the tails of our path, so it is not produced by such a matching in $C_{nk}$.

We now construct an involution $f$ on these two types of metacycle matchings, which both reverses signs and preserves weights.

Suppose we have some $\mu\mu$ in our string. Consider the first hop component of the form $H^a\mu\mu$. Note that $\widetilde M$ contains no edges in two cycles corresponding to the $\mu\mu$, by definition of a metacycle matching. Let $i$ be the vertex corresponding to the first $H$ and $j$ be the vertex corresponding to the first $\mu$. Then $f$ maps the matchings of cycles $i$ through $j-1$ to the cycles $i+1$ through $j$ and replaces the edge on the $C_k$ with the (unique) perfect matching of cycle $j+1$ not using the hop, and the perfect matching of cycle $i$ using the hop. This produces a new string where the $H^a\mu\mu$ is replaced with an $H^{a+1}r$ with the top and bottom rows completely matched. This pairs up the metacycle matchings not produced by applying the Cases above. Since it also changes the number of edges in the $C_k$ by exactly one, it is a sign-reversing, weight-preserving involution, as desired.

Hence for both odd and even $n$, the coefficients of $x^{nk-2m}$ in ($*$) match those of $\mathcal C_{kn}(x)$; thus the polynomials are equal.

\end{proof}

\section{Proof of the Conjecture of C. Hall}\label{section:Main}

A similar combinatorial proof can be used to show $$
	\mathcal{P}_{dn + n -1}(x) = \mathcal{P}_d (\mathcal{C}_n (x)) \mathcal{P}_{n-1}(x),
$$
where we replace the metacycle with a $P_d$ with an $n$ cycle corresponding to each vertex and an additional $P_{n-1}$ on the bottom. (Alternatively, the equivalent divisibility relation for Chebyshev polynomials is classical.)
Dividing both sides by $\mathcal P_{n-1}(x)$, we see that Conjecture \ref{conj:Hall} is equivalent to

\begin{theorem}\label{thm:main}
$\mathcal{C}_{n,d} (x) = \mathcal{P}_d (\mathcal{C}_n(x))$.
\end{theorem}

We do this by expanding the right- and left-hand sides into polynomials of $\mathcal{C}_n(x)$ and comparing coefficients. On the right-hand side, we simply use the definition:

$$ \mathcal{P}_d (\mathcal{C}_n(x)) = \sum_{m\geq 0} (-1)^m a(P_d,m)(\mathcal{C}_n(x))^{d-2m}. $$

To handle the left-hand side, we first recall Corollary \ref{Prop:WARMUP} from Section \ref{section:Comp}:

$$
	\mathcal{C}_{n,d}(x)= \frac{1}{d!}\sum_{\substack{\sigma\in S_d \\ \text{cyc}(\sigma)=\mu}} \prod\limits_{i=1}^{k} \mathcal{C}_{n\mu_{i}}(x)
$$

Then, using the result of Section \ref{subsection:Commute} and the definition of the matching polynomial, we obtain

\begin{align*}
	\mathcal{C}_{n,d}(x)
    &=\frac{1}{d!}\sum_{\substack{\sigma\in S_d \\ \text{cyc}(\sigma)=\mu}} \prod\limits_{i=1}^{k} \mathcal{C}_{\mu_{i}}(\mathcal{C}_n(x)) \\
	&=\frac{1}{d!}\sum_{\substack{\sigma\in S_d \\ \text{cyc}(\sigma)=\mu}} \prod\limits_{i=1}^{k} \left(\sum_{\ell \geq 0}(-1)^\ell a(C_{\mu_i},\ell)(\mathcal{C}_n(x))^{\mu_i-2\ell}\right).
\end{align*}

Expanding the inside product, we obtain a sum over compositions. Recalling that given a composition $\alpha$, $S(\alpha)=\alpha_1+\cdots+\alpha_n$, we have

\begin{align*}
	\mathcal{C}_{n,d}(x)
	&=\frac{1}{d!}\sum_{\substack{\sigma\in S_d \\ \text{cyc}(\sigma)=\mu}}\left(\sum_{\alpha=(\alpha_1,\dots,\alpha_{k})}(-1)^{S(\alpha)}\left(\prod_{i=1}^{k}a(C_{\mu_i},\alpha_i)\right)(\mathcal{C}_n(x))^{d-2S(\alpha)}\right) \\
    &=\frac{1}{d!}\sum_{\substack{\sigma\in S_d \\ \text{cyc}(\sigma)=\mu}}\sum_{m\geq 0}~\sum_{\alpha \vDash m}(-1)^{m}(\mathcal{C}_n(x))^{d-2m}\left(\prod_{i=1}^{k}a(C_{\mu_i},\alpha_i)\right).
\end{align*}

Notice that $\sum_{\alpha\vDash m}\prod_{i=1}^k a(C_{\mu_i},\alpha_i)$ is the number of matchings of $G_\sigma$ with $m$ edges, i.e. $a(G_\sigma,m)$. Substituting this into the sum eliminates the dependence of the summands on $\alpha$. In particular, we conclude that
$$
	\mathcal{C}_{n,d}(x) =\frac{1}{d!}\sum_{\sigma\in S_d}\left(\sum_{\alpha \vDash m}(-1)^{m}a(G_\sigma,m)(\mathcal{C}_n(x))^{d-2m}\right).
$$

Comparing coefficients with the right-hand side calculation reduces solving Theorem \ref{thm:main} to its combinatorial core, namely it now suffices to show $d!\, a(P_d,m)=\sum_{\sigma\in S_d}a(G_\sigma,m)$. 

\subsection{Bijective conclusion}

Given a permutation $\sigma$ on any set $S$, we define $G_\sigma$ as a directed graph having vertices $S$ and an edge from each $s$ to its corresponding $\sigma(s)$. Observe that such a graph may have two-cycles, and that the directedness is necessary to distinguish $G_\sigma$ from $G_{\sigma^{-1}}$.

% Bijections

\begin{prop}\label{BijGP}
$d!\, a(P_d,m) = \sum\limits_{\sigma\in S_d} a(G_\sigma,m)$
\end{prop}
\begin{proof}

The left-hand side counts diagrams which are formed by taking a permutation in two-line notation, drawing an rightward-oriented path down the middle, and choosing a matching on the graph.

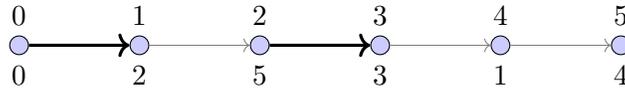
\begin{center}
			\begin{tikzpicture}[scale=.4,auto=left,every node/.style={circle,draw,minimum size=4.5pt,inner sep=2.5pt,fill=blue!20}]

			\node[draw=none,fill=none] at (0,2) {$0$};
			\node[draw=none,fill=none] at (4,2) {$1$};
			\node[draw=none,fill=none] at (8,2) {$2$};
			\node[draw=none,fill=none] at (12,2) {$3$};
			\node[draw=none,fill=none] at (16,2) {$4$};
			\node[draw=none,fill=none] at (20,2) {$5$};	

			\node (n1) at (0,0+1) {} ;
			\node (n3) at (4,0+1) {} ;
			\node (n6) at (8,1) {} ;
			\node (n4) at (12,0+1) {} ;
			\node (n2) at (16,0+1) {} ;
  			\node (n5) at (20,0+1) {} ;

			\node[draw=none,fill=none] at (0,0) {$0$};
			\node[draw=none,fill=none] at (4,0) {$2$};
			\node[draw=none,fill=none] at (8,0) {$5$};
			\node[draw=none,fill=none] at (12,0) {$3$};
			\node[draw=none,fill=none] at (16,0) {$1$};
			\node[draw=none,fill=none] at (20,0) {$4$};	
	             		
			\draw[very thick, ->] (n1) to (n3);
			\draw[gray,->] (n3) to (n6);
			\draw[very thick, ->] (n6) to (n4);
			\draw[gray, ->] (n4) to (n2);                		            		
			\draw[gray, ->] (n2) to (n5);
			\end{tikzpicture}
\end{center}

To see why this is counted by the right-hand side as well, remove the top row of labels. Then the collection of labels of the \textbf{left ends}, which are all vertices except for the rightmost vertices in edges of the matching. In the example above, the labels on the left ends are $\{0,1,4,5\}$. Now fill in the top row with the left end labels, in increasing order, only over the left ends.

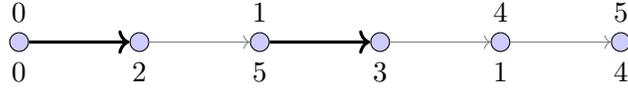
\begin{center}
			\begin{tikzpicture}[scale=.4,auto=left,every node/.style={circle,draw,minimum size=4.5pt,inner sep=2.5pt,fill=blue!20}]

			\node[draw=none,fill=none] at (0,2) {$0$};
			\node[draw=none,fill=none] at (4,2) {};
			\node[draw=none,fill=none] at (8,2) {$1$};
			\node[draw=none,fill=none] at (12,2) {};
			\node[draw=none,fill=none] at (16,2) {$4$};
			\node[draw=none,fill=none] at (20,2) {$5$};	

			\node (n1) at (0,0+1) {} ;
			\node (n3) at (4,0+1) {} ;
			\node (n6) at (8,1) {} ;
			\node (n4) at (12,0+1) {} ;
			\node (n2) at (16,0+1) {} ;
  			\node (n5) at (20,0+1) {} ;

			\node[draw=none,fill=none] at (0,0) {$0$};
			\node[draw=none,fill=none] at (4,0) {$2$};
			\node[draw=none,fill=none] at (8,0) {$5$};
			\node[draw=none,fill=none] at (12,0) {$3$};
			\node[draw=none,fill=none] at (16,0) {$1$};
			\node[draw=none,fill=none] at (20,0) {$4$};	
	             		
			\draw[very thick, ->] (n1) to (n3);
			\draw[gray,->] (n3) to (n6);
			\draw[very thick, ->] (n6) to (n4);
			\draw[gray, ->] (n4) to (n2);                		            		
			\draw[gray, ->] (n2) to (n5);
			\end{tikzpicture}
\end{center}

For the remaining vertices, give them the same label on the top row as on the bottom. Call the resulting permutation $\sigma'$. The permutation $\sigma$ that the right-hand sum keeps track of is $\sigma'\tau_M$, where $M$ is the matching and $\tau_M$ is the involution which fixes the vertices not involved in the matching, while swapping the left and right vertices of each edge. This $\sigma$ is chosen to ensure that $G_\sigma$ contains all directed edges of $M$.
        
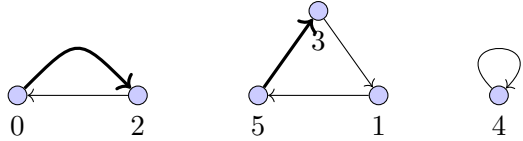
\begin{center}
			\begin{tikzpicture}[scale=.4,auto=left,every node/.style={circle,draw,minimum size=4.5pt,inner sep=2.5pt,fill=blue!20}]
			\node (n1) at (0,0+1) {} ;
			\node (n3) at (4,0+1) {} ;
			\node (n6) at (8,1) {} ;
			\node (n2) at (12,0+1) {} ;
           	\node (n4) at (10,2.81+1) {} ;
  			\node (n5) at (16,0+1) {} ;

			\node[draw=none,fill=none] at (0,0) {$0$};
			\node[draw=none,fill=none] at (4,0) {$2$};
			\node[draw=none,fill=none] at (8,0) {$5$};
			\node[draw=none,fill=none] at (10,2.81) {$3$};
			\node[draw=none,fill=none] at (12,0) {$1$};
			\node[draw=none,fill=none] at (16,0) {$4$};

         		\draw[->] (n3) to (n1);
			\draw[->, very thick] (n1)..controls (2,3) ..(n3);
			\draw[->,very thick] (n6) to (n4);
			\draw[->] (n4) to (n2);                		            			
			\draw[->] (n2) to (n6);
			\draw[->] (n5) .. controls (14,3) and (18,3) .. (n5);
			\end{tikzpicture}
\end{center}

The right-hand side clearly counts the total number of matchings with $m$ edges in $G_\sigma$ for all permutations $\sigma$, and this construction shows that we can obtain these from diagrams counted on the left-hand side. It is clear that no matching is counted by more than one diagram, since the matchings themselves must be identical, and therefore the two $\sigma'$, hence the two original permutations, must be identical. Moreover, we obtain every matching possible from this construction, since given a $\sigma$ and $M$, we obtain this from $\sigma'=\sigma\tau_M$ and $M$. From $M$ we can read off the collection of left ends, and thus recover the original permutation.

\end{proof}

The bijection in Proposition \ref{BijGP} is essentially a concatenation of two others (with a slight tweak for readability). We include these bijections for completeness.

\begin{prop}\label{BijKP}
$2^m(d-m)!\, a(K_d,m)=d!\, a(P_d,m)$.
\end{prop}
\begin{proof}

As before, the right-hand side counts diagrams which are formed by taking a permutation in two-line notation, drawing an rightward-oriented path down the middle, and choosing a matching on the graph. Equivalently, it counts paths with double labels on each of the left ends, where each of the lower labels of a left end is the upper label of some (usually different) left end.

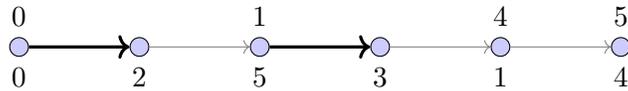
\begin{center}
			\begin{tikzpicture}[scale=.4,auto=left,every node/.style={circle,draw,minimum size=4.5pt,inner sep=2.5pt,fill=blue!20}]

			\node[draw=none,fill=none] at (0,2) {$0$};
			\node[draw=none,fill=none] at (4,2) {};
			\node[draw=none,fill=none] at (8,2) {$1$};
			\node[draw=none,fill=none] at (12,2) {};
			\node[draw=none,fill=none] at (16,2) {$4$};
			\node[draw=none,fill=none] at (20,2) {$5$};	

			\node (n1) at (0,0+1) {} ;
			\node (n3) at (4,0+1) {} ;
			\node (n6) at (8,1) {} ;
			\node (n4) at (12,0+1) {} ;
			\node (n2) at (16,0+1) {} ;
  			\node (n5) at (20,0+1) {} ;

			\node[draw=none,fill=none] at (0,0) {$0$};
			\node[draw=none,fill=none] at (4,0) {$2$};
			\node[draw=none,fill=none] at (8,0) {$5$};
			\node[draw=none,fill=none] at (12,0) {$3$};
			\node[draw=none,fill=none] at (16,0) {$1$};
			\node[draw=none,fill=none] at (20,0) {$4$};	
	             		
			\draw[very thick, ->] (n1) to (n3);
			\draw[gray,->] (n3) to (n6);
			\draw[very thick, ->] (n6) to (n4);
			\draw[gray, ->] (n4) to (n2);                		            		
			\draw[gray, ->] (n2) to (n5);
			\end{tikzpicture}
\end{center}

To see why these are counted by the left-hand side as well, again remove the upper labels. Remove the orientations for all edges not in the matching, and fill in the mising edges to interpret the matching as a directed matching of $K_d$.

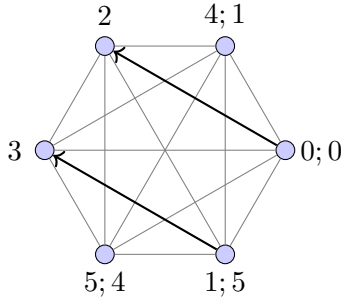
\begin{center}
			\begin{tikzpicture}[scale=.4,auto=left,every node/.style={circle,draw,minimum size=4.5pt,inner sep=2.5pt,fill=blue!20}]

			\foreach \a in {0,60,...,300} { %\a is the angle variable
			\node (n\a) at ($(\a:4cm)+(12,3)$) {}; 
			}
			\foreach \from/\to in {0/60,0/180,0/240,0/300,60/120,60/180,60/240,60/300,120/180,120/240,120/300,180/240,180/300,240/300} { %\a is the angle variable
			 \draw[gray] (n\from) to (n\to);
			}
			\foreach \from/\to in {0/120,300/180} { %\a is the angle variable
			 \draw[thick,->] (n\from) to (n\to);
			}

			\node[draw=none,fill=none] at ($(n0)+(1+0.2,0-0.1)$) {$0;0$};
			\node[draw=none,fill=none] at ($(n60)+(0,1)$) {$4;1$}; 
			\node[draw=none,fill=none] at ($(n120)+(0,1)$) {$2$}; 
			\node[draw=none,fill=none] at ($(n180)+(-1,0)$) {$3$}; 	
			\node[draw=none,fill=none] at ($(n240)+(0,-1)$) {$5;4$}; 	
			\node[draw=none,fill=none] at ($(n300)+(0,-1)$) {$1;5$}; 	
			\end{tikzpicture}
\end{center}

The left-hand side clearly counts the total number of directed matchings with $m$ edges on $K_d$ together with a permutation of the vertices which are not targets (arrow points) of the edges. This construction shows that such objects can be constructed from the diagrams counted by the right-hand side. There is an inverse construction as follows: ignore the bottom label (or the second label in the example object above). Starting with vertex 1, any time we are at the source of an edge in the directed matching, go to the target, and otherwise we go to the smallest-labeled vertex we have not gone to yet.
\end{proof}

\begin{prop}\label{BijGK}
$2^m (d-m)!\, a(K_d,m) = \sum\limits_{\sigma\in S_d} a(G_\sigma,m)$.
\end{prop}
\begin{proof}
As before, the left-hand side counts directed matchings $M$ of $K_d$ together with a permutation on the vertices which are not targets of the edges in the matching. Extend the permutation to a permutation $\sigma'$ on all the vertices, trivially in the sense that each target is sent to itself. Define $\sigma=\sigma'\tau_M$, where $\tau_M$ is defined as in Proposition \ref{BijGP}; as before $G_\sigma$ contains $M$ among its edge set, and thus we have obtained a matching counted by the right-hand side.

The inverse construction is clear: given $\sigma$ and matching $M$ of $G_\sigma$, let $\sigma'=\sigma\tau_M$, interpret $M$ as a directed matching of $K_d$, and remove the second (necessarily duplicate) label on the targets.
\end{proof}

\bigskip
\bibliographystyle{plain}

\bibliography{myreferences}

\begin{thebibliography}{1}

\bibitem{DF}
David~Steven Dummit and Richard~M Foote.
\newblock {\em Abstract algebra}, volume~3.
\newblock Wiley Hoboken, 2004.

\bibitem{Godsil}
C.~Godsil.
\newblock {\em Algebraic Combinatorics}.
\newblock Chapman Hall/CRC Mathematics Series. Taylor \& Francis, 1993.

\bibitem{Hall}
Chris Hall, Doron Puder, and William~F. Sawin.
\newblock Ramanujan coverings of graphs.
\newblock {\em Advances in Mathematics}, 323:367 -- 410, 2018.

\bibitem{Heilmann}
Ole~J Heilmann and Elliott~H Lieb.
\newblock Theory of monomer-dimer systems.
\newblock In {\em Statistical Mechanics}, pages 45--87. Springer, 1972.

\bibitem{Marcus}
Adam Marcus, Daniel~A Spielman, and Nikhil Srivastava.
\newblock Interlacing families {I}: Bipartite {R}amanujan graphs of all
  degrees.
\newblock In {\em Foundations of Computer Science (FOCS), 2013 IEEE 54th Annual
  Symposium on}, pages 529--537. IEEE, 2013.

\bibitem{serre1980trees}
Jean~Pierre Serre.
\newblock {\em Trees Translated from the French by John Stillwell}.
\newblock Springer, 1980.

\bibitem{Baddad}
Daniel Walton.
\newblock A tiling approach to {C}hebyshev polynomials.
\newblock {\em Senior Thesis, Harvey Mudd College, Claremont, CA}, 2007.

\end{thebibliography}

\end{document}